\newif\ifpictures
\picturestrue

\documentclass[12pt, a4paper]{amsart}
\usepackage[left=3cm,right=3cm,top=3cm,bottom=3cm]{geometry}

\usepackage[english]{babel} 
\usepackage[utf8]{inputenc} % Required for inputting international characters
\usepackage[T1]{fontenc} % Output font encoding for international characters
\usepackage{xcolor,colortbl}

\usepackage{amsmath}
\usepackage{amssymb}
\usepackage{amsthm}
\usepackage{amscd}
\usepackage{amsfonts}
\usepackage{mathrsfs}
\usepackage{mathtools}
\usepackage{bbm}
\usepackage{bm}

\usepackage{pgf,tikz}
\usepackage{tikz-cd}
\usetikzlibrary{matrix,arrows,calc}
\usepackage{pgfplots}
\usepackage{graphicx}
\usepackage{framed}

\usepackage{todonotes}

\newcommand{\bR}{\mathbb{R}}
\newcommand{\R}{\mathbb{R}}

\newcommand{\sA}{\mathcal{A}}
\newcommand{\sB}{\mathcal{B}}

\newcommand{\cA}{\mathcal{A}}
\newcommand{\cB}{\mathcal{B}}

\newcommand{\e}{\mathrm{e}}

\newcommand{\balpha}{{\bm{\alpha}}}
\newcommand{\bbeta}{\bm{\beta}}

\newcommand{\blambda}{\bm{\lambda}}

\newcommand{\bnu}{\bm{\nu}}
\newcommand{\ba}{\bm{a}}
\newcommand{\bb}{\bm{b}}
\newcommand{\bc}{\bm{c}}

\newcommand{\bq}{\bm{q}}
\newcommand{\bu}{\bm{u}}
\newcommand{\bv}{\bm{v}}

\newcommand{\bx}{\bm{x}}
\newcommand{\by}{\bm{y}}

\newcommand{\bzero}{\mathbf{0}}
\newcommand{\SAGE}{\mathrm{SAGE}}

\newcommand{\obalpha}{{\overline{\balpha}}}

\newtheorem{theorem}{Theorem}[section]
\newtheorem{satz}[theorem]{Theorem}
\newtheorem{prop}[theorem]{Proposition}

\newtheorem{lemm}[theorem]{Lemma}
\theoremstyle{definition}

\newtheorem{bsp}[theorem]{Example}

 \DeclareMathOperator{\supp}{supp}
 
 \DeclareMathOperator{\conv}{conv}

 \DeclareMathOperator{\cl}{cl}

 \DeclareMathOperator{\AGE}{AGE}

 \DeclareMathOperator{\Pow}{Pow} % power cone

\newtheorem{remark}[theorem]{Remark}

\makeatletter
\def\@settitle{\begin{center}%
		\baselineskip14\p@\relax
		\bf\Large%<- NEW
		\@title
	\end{center}%
}
\makeatother

\title{Nonnegativity of polynomials with Newton simplex
  over $\mathcal{A}$-convex sets}
 
\author{G. Averkov}

\address{Gennadiy Averkov, Brandenburgische Technische Universit\"at Cottbus-Senften\-berg, Platz der Deutschen Einheit 1, 03046 Cottbus}

\email{averkov@b-tu.de}
 
\author{J. Ellwanger}

\address{Jonas Ellwanger, Goethe Universit\"at, Robert-Mayer Stra\ss{}e 10, 60325 Frankfurt am Main,
 Germany\medskip}

\email{ellwange@math.uni-frankfurt.de}

\author{T. Theobald}

\address{Thorsten Theobald, Goethe Universit\"at, Robert-Mayer Stra\ss{}e 10, 60325 Frankfurt am Main,
 Germany\medskip}

\email{theobald@math.uni-frankfurt.de}

\author{T. de Wolff}

\address{Timo de Wolff, Technische Universit\"at Braunschweig, Institut f\"ur Analysis und Algebra, AG Algebra, Universit\"atsplatz 2, 38106 Braunschweig,
 Germany\medskip}

\email{t.de-wolff@tu-braunschweig.de}

\subjclass[2020]{12D10, 14P05, 52A20, 90C23, 05E14}
\keywords{Nonnegativity, exactness, polynomials, simplex, SAGE, SONC, $\cA$-convexity}

\begin{document}

\begin{abstract}We study a class of polynomials whose positive support is
	the set of vertices of a simplex and which may have several
	negative support points in the simplex. Various groups of authors
	have provided an exact characterization for the nonnegativity of 
	a polynomial on $\R_+^n$ in this class 
	in terms of nonnegative circuit polynomials
	and that characterization provides a tractable nonnegativity test.
	We generalize this characterization to the constrained nonnegativity over a
	set $X$ under an additional convexity precondition in the moment space.
	This provides a tractable nonnegativity test over $X$ for the class
	in terms of a power cone program. Our proof employs a 
	convex reparameterization of the problem and a  
	separation argument from convex geometry. We also connect our result
	to the convex cone of constrained SAGE polynomials (sums of 
	arithmetic-geometric exponentials).
\end{abstract}

% \date{\today}
\maketitle

\section{Introduction}

We consider real polynomials over $\R_+^n$
where we also allow the exponent vectors to be real vectors and 
$\R_+$ denotes the set of nonnegative real numbers. 
Formally, let $\sA \subseteq \bR^n$ be a finite set.
Then a polynomial $p$ in this sense is of the form
\begin{equation}
  \label{eq:f1}
	p(\bx) = \sum_{\balpha \in \sA} c_{\balpha} \bx^{\balpha}
\end{equation}
with coefficients $c_{\balpha} \in \bR$. The reason for allowing
real exponent vectors is that this facilitates to view our framework also
from the perspective of \emph{signomials supported on $\mathcal{A}$},
i.e., functions of the form
\begin{equation}
  \label{eq:signomial1}
  f(\by) = \sum_{\balpha \in \sA} c_{\balpha} \exp \langle \balpha,\by \rangle
  \ ,
\end{equation}
which we obtain from~\eqref{eq:f1} by the substitution $x_i = \exp(y_i)$ and
where $\langle \cdot,\cdot \rangle$ denotes the usual scalar product.
We are interested in the question whether the polynomial $p$ 
in~\eqref{eq:f1} is nonnegative over a set $X \subseteq \R_{>0}^n$,
where $X$ is given in some reasonable way, say, by a system of explicit inequality constraints. 
Using continuity and additional function values $+\infty$ and 
$-\infty$, 
we can also allow to consider
$X \subseteq \R_+^n$ even for possibly negative exponents.
In what follows, we assume that the feasibility region $X$ is convex in a certain non-linear coordinate system (details follow). 

In recent years, investigating sparse settings became an intensely researched area in real algebraic geometry and in polynomial optimization.
For given support $\sA$, one considers the space $\R^\sA$ of all real multivariate polynomials supported on $\sA$. The more usual notation for $\R^{\cA}$ would be the space of all vectors 
$(c_\balpha)_{\balpha \in \sA}$, where $c_\balpha \in \R$, 
 but we use  this notation to denote the respective polynomial $p(\bx)$ defined by these coefficient vectors. An initiating moment for modern developments in sparse algebraic geometry was the celebrated BKK theorem, which provides a sparse counterpart of the classical B\'ezout's theorem,  and the subsequent 
 developments by Gelfand, Kapranov and Zelevinsky, which are summarized in \cite{Gelfand:Kapranov:Zelevinsky}, who coined the terminology ``$\sA$-philosophy'', and who investigated the behavior of the space $\R^\sA$ and structures within like $\sA$-discriminants.
Developments on the real algebraic geometry side include the fewnomial theory, initiated by Khovanski,  see e.g., \cite{rojas-2025} for an overview of the recent advances, 
and exploiting sparsity in a 
sums of squares context, see e.g., \cite{wang-magron-lasserre-tssos}, as well as the development of other certificates which perform well in sparse settings; see e.g., \cite{aps-2024}.
This development is further motivated by the fact that in many real-world applications 
the support sets $\sA$ of polynomials are sparse.
As solving polynomial optimization problems is NP-hard even in the unconstrained case and notoriously hard to solve in practice, characterizing classes of sparse polynomials 
for which nonnegativity on a convex 
set $X$ can be decided efficiently and effectively is of ubiquitous interest.

In the case of global nonnegativity of usual polynomials (i.e.,
with integer exponents), a prominent classical 
result of this kind includes Hilbert's classification
\cite{hilbert-1888}. 
For homogeneous polynomials in two variables, homogeneous 
quadratic forms or homogeneous ternary quartics, this classification
ensures the equality of nonnegative polynomials
with sums of squares. From the viewpoint of convex optimization, 
the global nonnegativity problem can then be formulated as a 
semidefinite program. 
These techniques, however, do not extend to the case of 
real exponents.

\medskip

For a given set $\sA \subseteq \R^n$ of exponent vectors of a polynomial, we say that a set $X \subseteq \R^n_{+}$ is \emph{$\mathcal{A}$-convex} if
there exists $\obalpha \in \cA$ such that the image $\psi(X)$ of $X$ under the map $\psi : \R^n_{+} \to \R_{+}^n$, 
\[
	\psi(\bx) :=  \bigl(\bx^{\balpha-\obalpha} \bigr)_{\balpha \in \cA \setminus \{\obalpha\}}
\]
is a convex set. If ${\bf 0} \in \cA$ and we choose $\obalpha = {\bf 0}$, then the underlying map is simply 
\[
	\psi(\bx) := \bigl( \bx^{\balpha} \bigr)_{\balpha \in \cA \setminus \{\bf 0\}} \, ,
\]
that is, by this $\psi$, we carry out a monomial substitution. 

In this article, we consider the specific case that the Newton polytope of $p$ (that is, the convex hull $\conv(\sA)$ of
its support vectors $\sA$) is a simplex and $X$ is $\mathcal{A}$-convex. We assume that terms corresponding to non-vertices of $\conv(\cA)$ have negative coefficients, and all of the coefficients corresponding 
to vertices of $\conv(\cA)$ are either positive or unrestricted in sign. 
\medskip

Let us investigate the following toy example.
\begin{bsp} \label{ex: introduction}
	Consider the polynomial
	\begin{align*}
		p(x,y)=13+x^4y^2+x^2y^4-12xy-3x^2y^2, 
	\end{align*}
	whose support is shown in Figure~\ref{fi:newton-polytope3} (left). Set 
    $\mathcal{A} = \{(0,0)^T, (4,2)^T,(2,4)^T\}$.
	We aim to decide whether $p$ is nonnegative over the triangle  $X=\conv\{(0,0)^T,(1,1)^T,$ $(7/6,5/6)^T\}$.
	It turns out that $X$ is $\mathcal{A}$-convex.
	The respective convex set $\psi(X)$ for $\obalpha = {\bf 0}$ is depicted in Figure~\ref{fi:newton-polytope3} (right). 
	The nonnegativity criterion for this choice of the support and $X$ will be provided in Example~\ref{ex:solution-example}.
\end{bsp}

\ifpictures
\begin{figure}[t]
	\[\scalebox{0.8}{
	\begin{tikzpicture}
		\begin{axis}[
			x=1cm,y=1cm,
			axis lines=middle,
			xmin=0,
			xmax=4.5,
			ymin=0,
			ymax=4.5]
		\end{axis}
		\draw [fill=blue] (4,2) circle (2.5pt);
		\draw [fill=green] (2,2) circle (2.5pt);
		\draw [fill=blue] (2,4) circle (2.5pt);
		\draw [fill=green] (1,1) circle (2.5pt);
		\draw [fill=blue] (0,0) circle (2.5pt);
		\draw (0,0) -- (4,2);
		\draw (4,2) -- (2,4);
		\draw (0,0) -- (2,4);
	\end{tikzpicture}
}
		\hspace*{2cm}
\scalebox{0.7}{
	\begin{tikzpicture}
		\begin{axis}[
			axis lines=middle,
			xmin=-0.1,xmax=1.4,ymin=-0.1,ymax=1.4,
			xtick distance=1,
			ytick distance=1,
			xlabel=$$,
			ylabel=$$,
			grid=major,
			grid style={thin,densely dotted,black!20}]
			\draw plot[domain=1:7/6,smooth] 
              (axis cs:{\x^4*(1-1*(\x-1))^2},{\x^2*(1-1*(\x-1))^4});
			\draw plot[domain=0:7/6,smooth] 
              (axis cs:{\x^4*(5/7*\x)^2},{\x^2*(5/7*\x)^4});
			\draw plot[domain=0:1,smooth] 
              (axis cs:{\x^4*\x^2},{\x^2*\x^4});
		\end{axis}
	\end{tikzpicture}}
\]

	\caption{The support and the 
	Newton polytope of $p$  (left)
	and the convex set $\psi(X)$ for $\obalpha = 0$ (right) 
	in Example~\ref{ex: introduction}. 
	The depicted
	boundary of $\psi(X)$ consists of two line segments and one 
	nonlinear curve segment.
	} 
	\label{fi:newton-polytope3}
\end{figure}
\fi

\noindent
\textbf{Results.} In this paper, we provide an exact characterization of nonnegativity of  polynomials on $X$ in the setting described above.
In the unconstrained situation $X = \R^n_+$, the question has been solved in various variants by Iliman and de Wolff \cite{iliman-dewolff-resmathsci}, Murray, Chandrasekaran and Wierman \cite{mcw-newton-poly}, and by Wang \cite{wang-siaga-2022}.
All these papers build on methods for nonnegativity certificates based on the arithmetic-geometric inequality (AM/GM inequality), which define a full-dimensional convex subcone of the cone of nonnegative polynomials, also known as SONC or SAGE cone; see \cite{iliman-dewolff-resmathsci, chandrasekaran-shah-2016}.

Here, we employ a convex reparameterization of the problem
underlying those works and extend the result to the case 
of $\mathcal{A}$-convex sets $X$.
This leads to the following main theorem.

\begin{satz}\label{th:main}
	Let $\sA\subseteq \bR^n$ be the vertex set of a simplex, $\sB\subseteq \conv(\sA)\setminus \sA$ finite and
	$X \subseteq \R^n_{+}$ be $\mathcal{A}$-convex. 
	Then the 
	polynomial 
	\begin{equation}\label{eq:polynomialmain}
		p \ = \ \sum_{\balpha \in \sA}c_{\balpha}\bx^{\balpha}+\sum_{\bbeta \in \sB}d_{\bbeta }\bx^{\bbeta}
	\end{equation}
	 with $c_{\balpha} \in \R$ and $d_{\bbeta} < 0$
	is nonnegative on $X$ if and only if 
	$p$ has a decomposition of the form
	\begin{align}\label{eq:decompmain}
	p \ = \ \sum_{\bbeta \in \mathcal{B}} p_{\bbeta} \;
	\text{ with } p_{\bbeta}(\bx) = \sum_{\balpha \in \cA} c^{(\bbeta)}_{\balpha} 
	\bx^{\balpha} + 
	d_{\bbeta} \bx^{\bbeta}
	\end{align}
	with coefficients $c_{\balpha}^{(\bbeta)} \in \R$
	for all $\bbeta \in \mathcal{B}$ such that each $p_{\bbeta}$ is nonnegative
	over $X$.
Furthermore, if $X = \R^n_{+}$,
then the assertion remains true with the 
coefficients $c_{\balpha}$ in the assumption and the coefficients
$c_{\balpha}^{(\bbeta)}$ in the assertion restricted to be nonnegative. 
\end{satz}

Note that the specific situation $X=\R^n_+$ is the known unconstrained case as considered above and then, for a nonnegative polynomial, necessarily all 
$c_{\balpha}$ and $c^{(\bbeta)}_{\balpha}$  for $\balpha \in \cA$ are nonnegative. These implications do not hold
any more in the restricted situations.
As a consequence of the unconstrained case, deciding whether a polynomial in this class is nonnegative over $\R^n_+$ can be formulated as a convex optimization problem, specifically, a relative entropy program.
In the constrained case, we present 
a convex formulation 
in the moment space in Section~\ref{se:nonneg}.
Furthermore, the result reveals a hidden convexity structure in the nonnegativity question of our class of constrained polynomials.
Note that the hypothesis that $X$ is $\mathcal{A}$-convex is a significant restriction.

In Section~\ref{se:sage} we establish connections to the SAGE cone
and use this perspective to discuss stronger results in special cases 
as well as the frontiers of our main result. In particular, we
show that Theorem~\ref{th:main} coincides with the constrained
SAGE characterization in the univariate case 
(Theorem~\ref{th:a-positive-coefficients}) and
we complement Theorem~\ref{th:main} and the constrained SAGE characterization
by showing that the condition $\sB \subseteq \conv(\sA)\setminus \sA$ indeed is necessary via providing a counterexample to a more general statement; see 
Example \ref{ex:one-dim1}.

Our result strongly exploits the sign information of the support points.
In the study of real polynomials, using sign information of the coefficients
has a rich history dating back, in particular, to Descartes' rule of signs.
Recently, Bihan and Dickenstein \cite{bihan-dickenstein-descartes-rule} and Feliu and Telek \cite{feliu-telek-2022,telek-2024} generalized Descartes' rule of signs.

\section{Preliminaries\label{se:prelim}}
The current study was motivated by the question whether the
nonnegativity result for a polynomial with Newton simplex on $\R_+^n$ 
can be generalized to a constrained version where $X$ is a subset
of $\R_+^n$. When working
over constrained subsets $X$ of $\R_+^n$, a straightforward
generalized formulation of the result from
\cite{iliman-dewolff-resmathsci, mcw-newton-poly,wang-siaga-2022}
to the case of arbitrary sets $X$ or even to convex sets $X$ 
does not hold in general. This is a consequence of 
Example \ref{ex:counterex1} below.
Before we formalize that counterexample, we set up our notation
on sparse polynomials.

From now on, 
$\sA \subseteq \R^n$ is a nonempty, affinely independent, finite set
and $\sB \subseteq \R^n$ is nonempty and finite. Usually $\cA$ and $\sB$ are disjoint, the coefficients of $\sA$ are unrestricted in sign and the coefficients of $\sB$ are negative.
We denote by
$\bR^{\sA}$ the $|\sA|$-tuples of $\bR$ indexed by $\sA$.
The set $X$ is a fixed $\cA$-convex subset of $\bR^n_{+}$ and we
tacitly assume $X$ to be nonempty.

The following example shows that the $\cA$-convexity assumption on $X$ in 
Theorem~\ref{th:main} is crucial and that its assertion is not true for general sets $X$.

\begin{bsp} \label{ex:counterex1}
Consider the convex set $X = \R_{+} \times \{1\}$ and 
$\mathcal{A} =\{ \balpha^{(0)}, \balpha^{(1)}, \balpha^{(2)} \}
= \{(0,0)^T,(2,1)^T,(4,0)^T\}$,
$\mathcal{B} = \{\bbeta^{(1)}, \bbeta^{(2)} \} = 
\{(1,0)^T,(3,0)^T\}$. That is, we consider
polynomials of the form
\begin{equation}
  \label{eq:5-terms}
  p(x_1,x_2) = c_1 + c_2 \bx^{\balpha^{(1)}} + c_3 \bx^{\balpha^{(2)}}
    + d_1 \bx^{\bbeta^{(1)}}+ d_2 \bx^{\bbeta^{(2)}}\, .
\end{equation}
Then $p$ is nonnegative on $X$ if and only if the 
univariate polynomial
\begin{equation}
  \label{eq:5-terms2}
  q(x_1) := p(x_1,1) = c_1 + d_1 x_1 + c_2 x_1^2 + d_2 x_1^3 + c_3 x_1^4
\end{equation}
is nonnegative over $\R_{+}$. 
For every choice of two positive real numbers
$z_1 < z_2$,
consider the nonnegative univariate polynomial
$q(x_1) = (x_1 - z_1)^2 (x_1 - z_2)^2$,
which has a double zero in $z_1$ and in $z_2$ and 
which is of the form~\eqref{eq:5-terms2}. Denoting the
coefficients of $q$ by $c_1,c_2,c_3,d_1,d_2$ (as in~\eqref{eq:5-terms2}), 
then defining the polynomial $p$ as in~\eqref{eq:5-terms}
satisfies $p(x_1,1) = q(x_1)$. Assume that there exists a 
decomposition of the form~\eqref{eq:decompmain}, i.e., $p$
has a decomposition $p = p_1 + p_2$ such that
\begin{eqnarray*}
p_1 & = &  p_1(x_1,x_2) \ = \ c^{(1)}_1 + c^{(1)}_2 \bx^{\balpha^{(1)}} 
  + c^{(1)}_3 \bx^{\balpha^{(2)}} + d_1 \bx^{\bbeta^{(1)}} \, , \\
p_2 & = &  p_2(x_1,x_2) \ = \ c^{(2)}_1 + c^{(2)}_2 \bx^{\balpha^{(1)}} 
  + c^{(2)}_3 \bx^{\balpha^{(2)}} + d_2 \bx^{\bbeta^{(2)}}
\end{eqnarray*}
with coefficients $c_i^{(j)} \in \R$,
$d_i < 0$ and $p_1,p_2$ nonnegative over $X$.
Since the polynomials
$q_1(x_1):= p_1(x_1,1)$ and
$q_2(x_1):= p_2(x_1,1)$ must be nonnegative, the coefficients
of the lowest order nonzero coefficients and of the
highest order nonzero coefficients in $q_1$ and in $q_2$ must be positive.
Hence, each of $q_1$ and $q_2$ has at most two negative coefficients.
In the case of two negative coefficients, we consider the 
sequence of coefficients
sorted with respect to the order of the exponents.
In this sequence, the two negative coefficients are neighbored.

Both $q_1$ and $q_2$ must have zeros at $z_1$ and $z_2$.
Since $q_1$ and $q_2$ are nonnegative,
each of these two zeros must be of order at least two.
By Descartes' rule of signs (which also holds for the generalized
polynomials of the form~\eqref{eq:f1}),
the number of zeros on $\R_{>0}$, counting multiplicity,
is bounded by the number of sign changes in the sequence of coefficients. 
Since each of $q_1$ and $q_2$ has at most two sign changes in its sequence 
of coefficients, we obtain a contradiction.
Hence, there cannot exist a decomposition $p=p_1+p_2$
of the form~\eqref{eq:decompmain}.
\end{bsp}

\section{Nonnegativity of polynomials with Newton simplex\label{se:nonneg}}

In this section we prove our main Theorem~\ref{th:main} through
a convex reparameterization connected with a separation argument from 
convex geometry and discuss further examples and extensions.

Let $p$ be a polynomial of the form~\eqref{eq:polynomialmain}.
We can assume that $\bzero \in \mathcal{A}$. (Note that we do not
need to assume that $\conv(\mathcal{A})$ is full-dimensional.)
Let $Y$ be the moment space of $X$ with respect to the exponents $\sA$, i.e., the image of $X$ under the moment map
$\psi:\bx \mapsto (\bx^{\balpha} )_{\balpha \in \cA\setminus \{\bzero\}}.$

We now express the problem to decide whether $p$ is nonnegative on
$X$ in terms of the moment variables $v_{\balpha} = \bx^\balpha$. 
Note that $v_{\bzero} = 1$.
As described in terms of the mapping $\psi$, any valid choice of
$\bx \in X$ determines the vector 
$\bv = (v_{\balpha})_{\balpha \in \cA}$.
Conversely, the following consideration shows that we can consider
$p$ as a function in the variables 
$(v_{\balpha})_{\balpha \in \cA}$.
For every $\bbeta \in \conv(\cA)$, we can write $\bbeta$ in terms of the \textit{barycentric coordinates},
$\bbeta = \sum_{\balpha \in \mathcal{A}} \lambda^{(\bbeta)}_{\balpha} \balpha$
with $0 \le \lambda_{\balpha}^{(\bbeta)} \le 1$,
$\sum_{\balpha \in \mathcal{A}} \lambda_{\balpha}^{(\bbeta)} = 1$. 
Therefore every
\[
  \bx^{\bbeta} 
  = \prod_{\balpha \in \mathcal{A}} v_{\balpha}^{\lambda^{(\bbeta)}_{\balpha}} 
\]
is a function in the variables $(v_{\balpha})_{\balpha \in \cA}$. So we can write $p$ 
as a composition $p = F \circ \psi$ for a function $F$ defined on $Y$, namely, by setting
$F$ as
\begin{align}\label{eq:moment polynomial}
F =  \sum_{\balpha \in \sA} c_{\balpha} v_{\balpha}
 +\sum_{\bbeta \in \sB}d_{\bbeta}\prod_{\balpha\in \sA} v_{\balpha}^{\lambda_{\balpha}^{(\bbeta)}}.
\end{align}
Since $d_{\bbeta}<0$, the function \begin{align*}
	q_{\bbeta}=d_{\bbeta}\prod_{\balpha\in \sA} v_{\balpha}^{\lambda_{\balpha}^{(\bbeta)}}
\end{align*} is a convex function for each $\bbeta \in \sB$. 
For any convex function $f:Y \to \R$ on a convex set $Y \subseteq \R^k$ for some $k \ge 1$ let
\begin{align*}
	L_Y(f)=\{(a_0,\ba) \in \R \times \R^k \ | \ f(\by)\geq a_0 + \ba^T\by \;
     \text{for all } \by \in Y \}.
\end{align*}
We use the following lemma from convex analysis which can be seen
as a convex-geometric view of Fenchel's Duality Theorem
(see \cite[Theorem 31.1]{rockafellar-convex-analysis}). 

\begin{lemm}\label{lem:decomposition}
	Let $f_1,\ldots,f_m : Y \rightarrow \R$ be convex functions on a non-empty convex set $Y\subseteq \R^k$. Then $L_Y(\sum_{j=1}^{m}f_j)=\sum_{j=1}^{m}L_Y(f_j)$.
\end{lemm}

We provide an elementary proof based on a separation argument.

\begin{proof}
	We assume without loss of generality that $Y$ has non-empty interior,
    otherwise we let the ambient space be the affine hull of $Y$. 
    Moreover, it is sufficient to prove the case $m=2$. 
    The inclusion $\supseteq$ is clear. 
    For the other inclusion, let $(a_0,\ba) \in L_Y(f_1+f_2)$. 
    Then $f_1+f_2\geq a_0 + \ba^T \by$ on $Y$ and hence $f_1\geq a_0 + \ba^T \by-f_2$ on $Y$. 
    Since $a_0 + \ba^T\by$ is affine linear and $f_2$ convex, the function 
    $a_0 + \ba^T\by-f_2$ is a concave function. Therefore the sets
	\begin{align*}
		C\coloneqq \{(\by,z)\in Y \times \R \ | \ z\geq f_1(\by)\}
	\end{align*}
	and 
	\begin{align*}
		D \coloneqq \{(\by,z) \in Y \times \R \ | \ z\leq a_0 + \ba^T\by-f_2(\by)\}
	\end{align*}
	are convex, have nonempty interiors and do not share interior points. Hence there exists a hyperplane
    $H$ in $\R^k$ that separates $C$ from $D$. This hyperplane is non-vertical (i.e., not of the form 
    $\bq^T \by = \gamma$ for some nonzero vector $\bq$ and some constant $\gamma$),
    because verticality would imply $Y \subseteq H$ and thus contradict that $Y$ has nonempty interior. 
    Therefore,
    the hyperplane defines an affine function $b_0 + \bb^T\by$ that lies below $f_1$ and above $a_0 + \ba^T\by-f_2$ on $Y$.
     Hence
     $(b_0,\bb)\in L_Y(f_1)$, $(a_0,\ba)-(b_0,\bb)\in L_Y(f_2)$.
     Since $(a_0,\ba)=(b_0,\bb)+((a_0,\ba)-(b_0,\bb))$, we obtain $(a_0,\ba) \in L_Y(f_1)+L_Y(f_2)$.
\end{proof}
With this lemma we can now prove our main Theorem~\ref{th:main}.

\begin{proof}[Proof of Theorem \ref{th:main}]
	If $p$ has a decomposition of the form~\eqref{eq:decompmain}
    with the required properties
    then $p$ is clearly nonnegative over $X$. Conversely, 
	we look at our polynomial $p$ in the moment space. 
    By~\eqref{eq:moment polynomial}, we can write 
	\begin{align*}
		F = \sum_{\balpha \in \cA}c_\balpha v_\balpha+\sum_{\bbeta \in \sB} q_{\bbeta}
	\end{align*}
	with convex functions $q_{\bbeta}$ defined on the convex space 
	$Y:=\psi(X)$. 
	Set $\bc := (c_{\balpha})_{\balpha \in \sA}$.
	Since $p$ is nonnegative over $X$, the function $F$ is nonnegative over $Y$. Therefore we have $-\bc \in L_Y(\sum_{\bbeta \in \sB}q_{\bbeta})$. Using Lemma \ref{lem:decomposition} we know $-\bc \in \sum_{\bbeta \in \sB}L_Y(q_{\bbeta})$. So there exist vectors $\bc^{(\bbeta)}$ with 
	$\sum_{\balpha \in \cA} c_{\balpha}^{(\bbeta)} v_{\balpha} +q_{\bbeta}\geq 0$ 
	for $\bbeta \in \sB$ and $\sum_{\bbeta \in \sB}\bc^{(\bbeta)}=\bc$. Transforming back to the original coordinates yields the desired result.

To derive the special case $X=\R^n_+$ from the general result, we note that 
$\R_+^n$ is $\cA$-convex because $\psi(\R_{>0}^n) = \R_{>0}^{\mathcal{A} \setminus \{\mathbf{0}\}}$.
If a polynomial is nonnegative on $\R_+^n$ then its vertex coefficients are 
nonnegative (see \cite{chandrasekaran-shah-2016}). Applying
this property on $p$ and on the polynomials $p_{\bbeta}$ concludes the proof.
\end{proof}

\begin{bsp}\label{ex:counterex-convex}
For an example where the summands require some negative coefficient 
$c_{\balpha}^{(\beta)}$,
let $\mathcal{A} = \{(0,0,0)^T, (1,0,0)^T, (0,1,0)^T, (0,0,1)^T \}$,
$\mathcal{B} = \{(\frac{3}{5},\frac{1}{5},\frac{1}{5})^T, $
$  (\frac{1}{5},\frac{3}{5},\frac{1}{5})^T \}$ and
\[
p(\bx) \ = \ \frac{2}{5} + \frac{3}{5} x_3 
- \frac{1}{2} x_1^{3/5} x_2^{1/5} x_3^{1/5}
- \frac{1}{2} x_1^{1/5} x_2^{3/5} x_3^{1/5}.
\]
Note that the coefficients of the terms $\bx^{\balpha}$ are 0
for $\balpha \in \{(1,0,0)^T,(0,1,0)^T\}$. We consider
$X = \{\bx\in \bR_{+}^3 \mid x_1+x_2-x_3=1\}$.
Then the image of the set $X$ under the map 
$\bx \mapsto (\bx^{\balpha})_{\balpha \in \sA \setminus \{0\}}$
coincides with $X$ and hence $X$ is $\mathcal{A}$-convex.
The polynomial $p(\bx)$ is convex on $\R_+^{3}$.
Since $p(1,1,1) = 0$ and the gradient
$\nabla p(1,1,1) = (-2/5,-2/5,2/5)^T$ is proportional to the coefficients $(1,1,-1)^T$ of the
affine form defining $X$, we see that $(1,1,1)^T$ is the minimizer of $p$. Since $(1,1,1)^T$ is a root of $p$, the polynomial
$p$ is nonnegative on $X$. Indeed, a possible decomposition allowing negative coefficients 
is $p=p_1 + p_2$ with the nonnegative functions
\begin{eqnarray*}
  p_1 & = & \frac{1}{10} + \frac{1}{5} x_1 + \frac{1}{5} x_3 
    - \frac{1}{2} x_1^{3/5} x_2^{1/5} x_3^{1/5}, \\
  p_2 & = & \frac{3}{10} - \frac{1}{5}x_1 + \frac{2}{5} x_3 
    - \frac{1}{2} x_1^{1/5} x_2^{3/5} x_3^{1/5}
\end{eqnarray*}
on $X$. Now assume there were also a decomposition $p=p_1+p_2$ with
\begin{eqnarray*}
  p_1 & = & c_0^{(1)} + c_3^{(1)} x_3 
    -\frac{1}{2} x_1^{3/5} x_2^{1/5} x_3^{1/5}, \\
  p_2 & = & c_0^{(2)} + c_3^{(2)} x_3 
    - \frac{1}{2} x_1^{1/5} x_2^{3/5} x_3^{1/5}
\end{eqnarray*}
and nonnegative coefficients $c_0^{(1)}, c_3^{(1)}, c_0^{(2)}, c_3^{(2)}$. 
We obtain $\nabla p_1(1,$ $1,1) $ 
$= (-3/10, $ $-1/10, $ $c_3^{(1)} - 1/10)^T$, which shows that the gradient
cannot be proportional to $(1,1,-1)^T$ and thus $p_1$ cannot be a
nonnegative function on $X$
with a root at $(1,1,1)^T$.
\end{bsp}

Using the description in terms of the moment variables, the infimum 
of $p$ can be expressed in terms of a convex optimization
problem. For $\lambda_1, \ldots, \lambda_n > 0$ with 
$\sum_{i=1}^n \lambda_i=1$, we denote by 
$\Pow_{n+1}^{\lambda_1, \ldots, \lambda_n}$ a \emph{general power cone}
\[
  \Pow_{n+1}^{\lambda_1, \ldots, \lambda_n}
 \ = \ \{ (\bx,z) \mid
 x_1^{\lambda_1} \cdots x_n^{\lambda_n} \ge |z|, \,
 x_1, \ldots, x_n \ge 0 \},
\]
see, e.g., \cite{chares-thesis,mnt-2023,papp-2023}.
$\Pow_{n+1}^{\lambda_1, \ldots, \lambda_n}$ is a convex cone
in $(n+1)$-dimensional space.

The infimum of $p$ over $\R_+^n$ can be expressed in terms of the following
convex program, where $k = |\cA \setminus \{\bzero\}|$ and $\blambda^{(\bbeta)} = (\blambda^{(\bbeta)}_{\balpha})_{\balpha \in \cA}$.
\[
  \begin{array}{rcl@{\quad}l}
    \multicolumn{4}{l}{\inf\limits_{\bv = (v_{\balpha})_{\balpha \in \cA}, \,
    (z_{\bbeta})_{\bbeta \in \cB}} \, \sum\limits_{\balpha \in \cA} c_{\balpha} v_{\balpha} + \sum\limits_{\bbeta \in \cB} d_{\bbeta} z_{\bbeta}} \\
     ((v_{\balpha})_{\balpha \in \cA},
     z_{\bbeta}) & \in & \Pow_{k+2}^{\blambda^{(\bbeta)}} & \quad (\bbeta \in \cB) \, , \\
     v_{\bzero} & = & 1 \, .
  \end{array}
\]

If a description of the moment space $Y$ is known, then the decomposition can be computed by
expressing the infimum of $p$ over $X$ as the convex program
\begin{equation}
  \label{eq:convex-program1}
  \begin{array}{rcl@{\quad}l}
    \multicolumn{4}{l}{\inf\limits_{\bv = (v_{\balpha})_{\balpha \in \cA}, \,
    (z_{\bbeta})_{\bbeta \in \cB}} \, \sum\limits_{\balpha \in \cA} c_{\balpha} v_{\balpha} + \sum\limits_{\bbeta \in \cB} d_{\bbeta} z_{\bbeta}} \\
     ((v_{\balpha})_{\balpha \in \cA},
     z_{\bbeta}) & \in & \Pow_{k+2}^{\blambda^{(\bbeta)}} & \quad (\bbeta \in \cB) \, ,\\
     v_{\bzero} & = & 1 \, , \\
     (v_{\balpha})_{\balpha \in \cA \setminus \{\bzero\}} & \in & Y \, .   
  \end{array}
\end{equation}

\begin{bsp}\label{ex:solution-example}
	With the main theorem proven, we answer the question from Example \ref{ex: introduction} whether the polynomial
	\begin{align*}
		p(x,y)=13+x^4y^2+x^2y^4-12xy-3x^2y^2
	\end{align*}
	is nonnegative over 
	$X=\conv\{(0,0)^T,(1,1)^T,(7/6,5/6)^T\}$. 
	All the negative support points of $p$ are contained in the convex hull of the positive support points. Therefore, it suffices to calculate a decomposition as seen in Theorem \ref{th:main}. 
	For any $\gamma \in \bR$ such that $p-\gamma$ has a decomposition $p-\gamma=p_1+p_2$ with
	\begin{align*}
		&p_1=c_0^{(1)}+c_1^{(1)}x^4y^2+c_2^{(1)}x^2y^4-12xy \text{ and }\\ &p_2=c_0^{(2)}+c_1^{(2)}x^4y^2+c_2^{(2)}x^2y^4-3x^2y^2
	\end{align*} 
	and $c_j^{(1)}+c_j^{(2)}=c_j$ for $j\in \{0,1,2\}$, we can choose $p_1$ and $p_2$ to have the same global minimizer. The largest value of $\gamma$ for which $p-\gamma$ is globally nonnegative on 
    $\R_+^2$ is $\gamma=-7$. The resulting $p_1$ and $p_2$ are given by
	\begin{align*}
		&p_1(x,y)=16+ \frac{1}{2} x^4y^2+ \frac{1}{2} x^2y^4-12xy 
         \text{ and} \\ &p_2(x,y)=4+ \frac{1}{2} x^4y^2+ \frac{1}{2} x^2y^4-3x^2y^2.
	\end{align*} 
	These polynomials share the same global minimizer $(\sqrt{2},\sqrt{2})$. 
    By determining the function values at the critical points of $p_1$ and of $p_2$
    and by considering $p_1$ and $p_2$ on the boundary of the triangle $X$, 
    one can verify the nonnegativity of the polynomials $p_1-5$ and $p_2-2$,
    which gives the decomposition
    $p = (p_1-5)+ (p_2-2)$.
    The minimum of $p_1$, $p_2$ and $p$ over $X$ is attained at $(1,1)$.
	Hence, $\min_{(x,y)\in X}p(x,y)=0$ and thus $p$ is nonnegative on $X$. 
\end{bsp}

\section{Connections to the SAGE cone\label{se:sage}}

In this section we discuss the relation of Theorem~\ref{th:main} to 
the theory of constrained SAGE cones. To this end,
we address cases where the preconditions of the theorem and the 
preconditions for the use of the constrained SAGE cone coincide,
in particular, the univariate case.

\subsection{Background}
Let $p$ be a general polynomial whose coefficients except at most one are nonnegative,
\begin{equation}
	\label{eq:signomial5}
	p \ = \ \sum_{\balpha \in \sA} c_{\balpha} \bx^{\balpha} +
	d \bx^{\bbeta} \text{ with }
	c_{\balpha} \ge 0 \, \text{ and } d \in \bR.
\end{equation}

For $\bbeta \in \sB$, the set of all these functions which are nonnegative on $X$ forms a convex cone called the \emph{constrained $\bbeta$-$\AGE$ cone} with respect to $\sA$. We write
\begin{align*}
	C_X(\sA,\bbeta)=\left \{  p=\sum_{\balpha \in \sA} c_{\balpha} \bx^{\balpha} +d \bx^{\bbeta} \ | \ p \geq 0 \ \text{on} \ X \ \text{and} \  c_{\balpha}\geq 0 \ \text{for all} \ \balpha \in \sA \right\}.
\end{align*}
Based on $C_X(\sA,\bbeta)$ we also introduce the \emph{signed $X$-$\SAGE$ cone}
\begin{align}
    \label{eq:signed-sage-cone}
	C_X(\sA,\sB)=\sum_{\bbeta \in \sB}C_X(\sA,\bbeta),
\end{align}
where the sum denotes the Minkowski sum. By definition 
of $C_X(\sA,\sB)$, for every $p \in C_X(\sA,\sB)$ 
the exponent vectors corresponding to negative coefficients 
form a subset of $\sB$. Polynomials  in $ C_X(\sA,\sB)$
are called \emph{$X$-$\SAGE$-polynomials} with respect to $\sA$ and $\sB$. In this terminology, we omit $\sA$ and $\sB$ if their choice is clear from the context. 

In the following, let $X \subseteq \R_{>0}^n$ be a logarithmically
convex set, i.e., the set 
$\log(X)=\{ (\log x_1, \ldots, \log x_n) \, : \, \bx \in X \}$
is convex where $\log$ denotes the natural logarithm. 
(Note that in the signomial language of~\eqref{eq:signomial1} the set $X$ 
would transform to a convex set.)
As already mentioned in the introduction, by
considering function values $-\infty$ and $+\infty$ we can also
allow $X$ to be in the closure $\R_+^n$ of $\R_{>0}^n$.
 
The signed $X$-$\SAGE$ cone provides a convenient notation for expressing nonnegativity certificates based on  the AM/GM inequality and used in the context of optimization 
(see \cite{iliman-dewolff-lower-bounds,mnr-2022,mcw-newton-poly}).
Note that our definitions are
consistent with unsigned versions of the
$X$-SAGE cone in the literature (e.g., \cite{mcw-partial-dualization,mnt-2023,theobald-ragopt}).

We now explain how to decide membership to the $X$-SAGE cone.
The relative entropy function of two vectors 
$\bu,\bv\in \bR^n_{>0}$ is defined as $D(\bu,\bv)=\sum_{i=1}^{n}u_i\ln(u_i/v_i)$.
For a convex set $C$, denote 
by $\sigma_{C}(\by) = \sup\{ \langle \by, \bx \rangle \mid \bx \in C \}$ the 
\emph{support function} of $C$ from classical convex geometry. The function
$\sigma_{\log X}$ is a convex function $\bR^n \to \bR \cup \{\infty\}$.
If $\log X$ is a polyhedron, 
then $\sigma_{\log X}$ is linear on every normal cone of 
$\log X$.
In the following statement, $\e \bc$ is the usual scalar multiplication of Euler's number
$\e$ with the vector $\bc$. 

\begin{theorem} [\cite{chandrasekaran-shah-2016,mcw-partial-dualization}]~
	\label{th:nonneg-entropy-cond}
	\begin{enumerate}
		\item[1.]  The polynomial $p$ in~\eqref{eq:signomial5}
		is nonnegative on $\R^n_{+}$ 
		if and only if there exists $\bnu \in \R^{\sA}_+$
		with $\sum_{\balpha \in \sA}
		\nu_{\balpha} \balpha
		= (\sum_{\balpha \in \sA} \nu_{\balpha})\bbeta$
		and $D(\bnu,\e \bc) \le d$.
		\item[2.] Let $X$ be logarithmically convex.
        Then the polynomial $p$ in~\eqref{eq:signomial5}
		is nonnegative on $X$ if and only if there exists
		$\bnu \in \R^{\sA}_{+}$ with
		$\sigma_{\log X}(\sum_{\balpha \in \sA}\nu_{\balpha} (\bbeta-\balpha)) +
		D(\bnu,
		\e\bc) \le d$.
	\end{enumerate}
\end{theorem}

This statement can be used to express the convex cones $C_{X}(\sA,\bbeta)$ and 
$C_{X}(\sA,\sB)$ in terms of the relative entropy function and in terms of
the support function of $X$. As another consequence, if the sets $\cA$ and $\cB$
consist of rational vectors and the set 
$(\sA')^T \log X=\{ (\langle \balpha,\bx \rangle)_{\balpha \in \sA'}\in \bR^{\sA'} \  | \ \bx\in \log X\}$ 
with $\sA' := \sA \cup \{\bbeta\}$ 
is a rational polyhedron,
then, by \cite{mnt-2023}, the nonnegativity of $p$ on $X$ can be formulated
as a second-order cone program.
\begin{prop}[Theorem~1 and Corollary~1 in \cite{mcw-partial-dualization}]
  \label{pr:cx-rel-entropy}
    Let $X$ be logarithmically convex. Then
	the constrained $\bbeta$-AGE cone can be expressed as
	\begin{align*}
		C_{X}(\sA,\bbeta)= \bigg \{& p =\sum_{\balpha \in \sA} c_{\balpha} 
          \bx^{\balpha} +d \bx^{\bbeta} \ | \ 
           \text{There exists }\bnu \in \bR^{\sA}_+ \text{ with } \\ 
		&\sigma_{\log X}\left (\sum_{\balpha \in \sA}\nu_{\balpha} (\bbeta-\balpha)\right )+D(\bnu,\e\bc)\leq d
		 \text{ and }c_{\balpha}\geq 0\bigg \}.
	\end{align*}
\end{prop}

Combining Proposition~\ref{pr:cx-rel-entropy} with 
\eqref{eq:signed-sage-cone}
gives a characterization $C_X(\sA,\sB)$ in terms of the relative entropy function and the support function of $\log X$. 

\begin{remark}
The SAGE language, which we introduce here, has a counterpart called sums of nonnegative circuit functions.
The main difference is that building blocks there are circuits, i.e., minimally affinely dependent support sets with one negative term, but the two resulting cones are identical up to technicalities detailed by Wang \cite{wang-siaga-2022} and, independently, Murray, Chandrasekaran, and Wierman \cite{mcw-newton-poly}.
In the circuit setting, nonnegativity is decided by an invariant called circuit number, which is the counterpart of the entropy function in the SAGE language.
\end{remark}

\subsection*{Connections between Theorem~\ref{th:main} and the SAGE cone}
Since in our main setup, $\sA$ is the vertex set of a simplex, we have in the one-dimensional case that $|\sA|\leq 2$, w.l.o.g.\ we can assume $|\sA|=2$.

In the one-dimen\-sion\-al case, the set $X$ is $\mathcal{A}$-convex if
and only if $X$ is convex and, as another characterization,
if and only if $X$ is logarithmically convex.
The following theorem shows that in the
one-dimensional case with nonnegative coefficients $c_{\alpha}$, 
the coefficients $c_{\alpha}^{(\beta)}$ in
the decomposition of Theorem~\ref{th:main} can be restricted to be nonnegative.
As a consequence, for this class we obtain that the set of polynomials 
captured by the theorem coincides with the set of polynomials in
the corresponding $X$-SAGE cone.

\begin{satz}
\label{th:a-positive-coefficients}
Let $\mathcal{A} = \{\alpha_1,\alpha_2\} \subseteq \R$,
$\mathcal{B} \subseteq \conv(\cA) \setminus \cA$ and $X \subseteq \R_{+}$
convex.
Further let $p$ be a univariate polynomial of the
form~\eqref{eq:polynomialmain} with $c_{\alpha} \ge 0$ for 
$\alpha \in \cA$ and $d_\beta \le 0$ for $\beta \in \cB$.
Then $p$ is nonnegative over $X$ if and only if there exists a decomposition 
of the form as in 
Theorem~\ref{th:main} such that the coefficients $c_{\alpha}^{(\beta)}$ are
nonnegative for $\beta \in \cB$, that is,
\begin{equation}
  \label{eq:decomp-univariate}
  p \ = \ \sum_{\beta \in \mathcal{B}} p_{\beta} \;
  \text{ with }p_{\beta}(x) = \sum_{\alpha \in \cA} c^{(\beta)}_{\alpha} 
  x^\alpha + 
  d_{\beta} x^{\beta}
\end{equation}
with coefficients $c_{\alpha}^{(\beta)} \ge 0$
for all $\alpha \in \cA$ and $\beta \in \mathcal{B}$ and $p_{\beta}$
is nonnegative over $X$.
Hence, $p$ is nonnegative over $X$
if and only if $p$ is contained in
the $X$-SAGE cone $C_X(\cA,\cB)$.
\end{satz}

Note that in the hypothesis of the theorem we also allow the cases $c_{\alpha} = 0$
and $d_{\beta} = 0$.

\begin{proof}
Let $\mathcal{A} = \{\alpha_1,\alpha_2\}$ with $\alpha_1 < \alpha_2$.
The case $X= \{0\}$ is trivial, therefore we assume $X \neq \{0\}$.
Since $p$ is nonnegative over $X \cap \R_{>0}$ if and only if the polynomial
$x^{\gamma} p$ is nonnegative over $X \cap \R_{>0}$ for every $\gamma \in \R$, we can assume
that $\alpha_1 = 0$, i.e., $c_{\alpha_1}$ is the constant term of $p$
(Recall from the introduction that we allow $X \subseteq \R_+^{n}$ even
for negative exponents.)
Further, since $p$ is nonnegative over $X$ if and only if $p$ is nonnegative over the 
closure $\cl X$, we can assume without loss of generality that $X$ is closed.

If a decomposition of the form~\eqref{eq:decomp-univariate} 
exists, then $p$ is clearly nonnegative over $X$. Conversely,
let $\mathcal{P}$ be the cone of polynomials of the form \eqref{eq:polynomialmain} with $c_\alpha\geq 0$ and $d_\beta\leq 0$ which are nonnegative over $X$. This cone is convex, closed and pointed. Using Carath\'{e}odory's Theorem, it is sufficient to prove the claim for every generator of an extreme ray of this cone. Let $p$ be such a generator. 
If there exists at most one $\beta \in \sB$ with $d_\beta\neq 0$, then the statement is clearly true. 
So let there exist $\beta_1,\beta_2 \in \sB$ with $d_{\beta_1},d_{\beta_2}<0$. Using Descartes' rule of signs, $p$ can have at most two zeros on $X$ 
counted with multiplicity.
(If $0 \in X$ then there cannot be
a zero in the origin because
$d_{\beta_1}, d_{\beta_2} < 0$ and $X \neq \{0\}$
would imply $c_{\alpha_1} > 0$.)
We distinguish three possible cases: $p$ has no zeros on $X$,
$p$ has a double zero in $X$ or $p$ has only simple zeros in $X$.

If $p$ has no zero on $X$, the polynomial is strictly positive over $X$.  
Moreover, if the closed set $X$ is unbounded then
$c_{\alpha_2}$ must be strictly positive and thus $\lim_{x \to \infty} p(x) = \infty$.
Therefore replacing the coefficient $d_{\beta_1}$ by $d_{\beta_1}+\varepsilon$ or $d_{\beta_1}-\varepsilon$ for sufficiently small $\varepsilon > 0$ results in two polynomials $p_1,p_2 \in \mathcal{P}$. Since $p=\frac{1}{2}(p_1+p_2)$ 
and $p_1$ and $p_2$ are not proportional,
this contradicts that $p$ is a generator of an extreme ray of $\mathcal{P}$.

Now consider the case that $p$ has a double zero $z$ in $X$.
We have $z > 0$, because either $0 \not\in X$ or there cannot be
a zero in the origin by the earlier observation.
Then the properties 
$d_{\beta_1},d_{\beta_2} < 0$ in conjunction with Descartes' rule 
imply $c_{\alpha_1}, c_{\alpha_2} > 0$. In particular, 0 is not a zero of $p$
on $X$ and $\lim_{x \to \infty} p(x) = \infty$.
The double zero at $z$ implies that $p(z)=p'(z)=0$. We regard these conditions
as two linear equations in the variables 
$c_{\alpha_1}, c_{\alpha_2},d_{\beta_1},d_{\beta_2}$.
Since this system is underdetermined,
we can perturb the original values $c_{\alpha_1},c_{\alpha_2},d_{\beta_1},d_{\beta_2}$ 
using a perturbation vector $(s_1,s_2,t_1,t_2)$ 
with sufficiently small absolute components
such that all the signs of $(c_{\alpha_1},c_{\alpha_2},d_{\beta_1},d_{\beta_2})$ 
equal the signs of $(c_{\alpha_1},c_{\alpha_2},d_{\beta_1},d_{\beta_2}) 
\pm (s_1,s_2,t_1,t_2)$ 
and the vectors $(c_{\alpha_1},c_{\alpha_2},d_{\beta_1},d_{\beta_2}) 
\pm (s_1,s_2,t_1,t_2)$ 
are still solutions to the linear equations. The set of perturbation vectors
with this property is at least two-dimensional. Since 
$\lim_{x \to \infty} p(x) = \infty$
and $z$ is the only root on the closed set $X$,
for a sufficiently small perturbation we can obtain the coefficient vectors
of two polynomials $p_1, p_2 \in \mathcal{P}$ with 
$p=\frac{1}{2}(p_1+p_2)$ which are not proportional. Here,
the nonnegativity of $p_1, p_2$ follows from the representation
$p = (x-z)^2 g$ with some positive, continuous function 
$g$ on $X$.
This contradicts that $p$ is a generator of an extreme ray of $\mathcal{P}$.

Now consider the case that $p$ has only simple zeros in $X$. There can be at most two
of them and they have to lie on the boundary of $X$.
Either $X$ is compact or $X$ is of the form $[a,\infty)$ with 
$a \ge 0$ and $\lim_{x \to \infty} p(x) = \infty$. 
For each of the zeros $z_i$ consider the equation $p(z_i) = 0$. 
This gives at most two equations and we can
use a similar argument as in the case of a double root
based on perturbing the nonzero coefficients. In case 
$c_{\alpha_1} = 0$ (or $c_{\alpha_2} = 0$) the maximal number of
conditions decreases as well since there can be at most one
simple zero in $X$ then.
\end{proof}

We briefly consider the question whether Theorem~\ref{th:a-positive-coefficients}
can be generalized to situations where the 
set $\cB$ is possibly not contained in $\conv(\cA) \setminus \cA \subseteq \R$. 
This would also provide
a generalization of Theorem \ref{th:main} in the one-dimensional case where
positive support points in the interior of the Newton simplex are allowed.
If we only consider polynomials with at most one negative term, the nonnegativity 
cone over $X$ clearly coincides with the signed $X$-SAGE cone by definition.
The following counterexample with two negative terms shows that,
even if all $c_{\alpha} > 0$
for $\alpha \in \cA \subseteq \R$, the nonnegativity property of $p$
does not coincide with the existence of a 
decomposition of the form~\eqref{eq:decompmain}
when there is at least one $\beta\notin \conv(\sA)$.

\begin{bsp}
	  \label{ex:one-dim1}
	Let $\cA = \{1,3\}$ and $p = \sum_{i=0}^{3} c_i x^{i}$ with $c_0,c_2 \le 0$ and $c_1, c_3 > 0$ on the set $X=[1,\infty)$. The support points of $p$ are shown
	in Figure~\ref{fi:newton-polytope1}.
\ifpictures
\begin{figure}[t]
	\[
         \begin{tikzpicture}
             \begin{scriptsize}
                   \draw (-1,0) -- (4,0);
                   \draw [fill=green] (0,0) circle (2.5pt);
                   \draw[color=green] (0.2,0.2) node[black] {$\beta_1$};
                   \draw (0,-0.3) node {0};
                   \draw [fill=blue] (1,0) circle (2pt);
                   \draw[color=blue] (1.2,0.2) node[black] {$\alpha_1$};
                   \draw (1,-0.3) node {1};
                   \draw [fill=green] (2,0) circle (2.5pt);
                   \draw[color=green] (2.2,0.2) node[black] {$\beta_{2}$};
                   \draw (2,-0.3) node {2};
                   \draw [fill=blue] (3,0) circle (2.5pt);
                   \draw[color=blue] (3.2,0.2) node[black] {$\alpha_2$};
                   \draw (3,-0.3) node {3};
             \end{scriptsize}
             \end{tikzpicture}
\]

\caption{The support points of $p$ in Example~\ref{ex:one-dim1}}
\label{fi:newton-polytope1}
\end{figure}
\fi
We can assume $c_3 = 1$.
Consider a subfamily of nonnegative polynomials of the previous form
which have a double root, say, at $2$. This
family is a one-dimensional family, and parametrizing it in terms of
$c_2$ gives
$p = x^3 + c_2 x^2 + (-12 - 4 c_2) x + (16+4c_2) \, .$
For $c_2 \in (-5,-4)$, the signs of the coefficients are as desired 
and $p(1) > 0$.
Hence, for $c_2 \in (-5,-4)$, the polynomial $p$ describes a family 
of nonnegative polynomials satisfying the sign constraints.

However, none of the polynomials in this family has a decomposition
of the desired form (with $c^{(i)}_{\alpha} \in \R$).
If one of these polynomials had such a certificate,
then there were a decomposition $p=p_1 + p_2$ with
$\supp p_1 \subseteq \{1,2,3\}$ and negative coefficient 
for the exponent 2, as well as $\supp p_2 \subseteq \{0,1,3\}$ and
negative coefficient for the exponent 0. By Descartes' rule,
$p_2$ cannot be nonnegative with a double root at $2$.
Hence, for $c_2 \in (-5,-4)$, the polynomial $p$ with the double root
at $2$ does not have a decomposition of the desired form.
\end{bsp}

\begin{remark}
In Example~\ref{ex:counterex1} we gave an example which shows that 
Theorem~\ref{th:main} does not work when replacing $\mathcal{A}$-convexity 
by usual convexity. The set $X$ in the example is also logarithmically
convex and thus allows to relate the situation to the $X$-SAGE cone.

Indeed, from that perspective,
there is a fundamental limitation that rules out the possibility of 
using the $X$-$\SAGE$ cone for completely characterizing nonnegativity in 
this example. The $X$-$\SAGE$ cone for the example is second-order cone representable, 
but due to the fact that the choice of $z_1,z_2$ above is arbitrary, it follows 
from the main result in \cite{averkov2019optimal} that the cone of polynomials 
\eqref{eq:5-terms} that are nonnegative on $X$ is not second-order cone 
representable. That is, it is not only the $X$-$\SAGE$ that does not 
characterize nonnegativity in this setting, but any approach that stays 
within the
second-order cone programming paradigm. By the main result of 
\cite{averkov:scheiderer:2025}, it is possible to characterize nonnegativity 
in this example using linear matrix inequalities of size $3 \times 3$. 
\end{remark}

\section{Outlook}

In this work, we have provided an exact nonnegativity characterization for a class of polynomials whose
Newton polytope is a simplex. It remains a general open question to extend our and other existing classes
of polynomials regarding nonnegativity results.
We mention that the papers \cite{Forsgaard:deWolff:BoundarySONCCone,mcw-newton-poly} and, in the symmetric setting
\cite{mrt-2025}, study other classes of polynomials or signomials with respect to exactness 
in the unconstrained case. One specific
question is whether those results also extend to constrained settings.

\section*{Acknowledgments}

The authors were supported by the DFG grants 539867386 (G.A.)
and 539847176 (J.E., T.T., T.d.W.)
in the context of the 
Priority Program SPP 2458 ``Combinatorial Synergies''.

\bigskip

\noindent
{\bf Declaration of generative AI and AI-assisted technologies in
the manuscript preparation process.} After the preparation of
a minor revision of the manuscript in an advanced stage of the
submission process, the authors used Claude Opus 5.5 for a
consistency check of the manuscript. After using this tool,
the authors reviewed and edited the content as needed and take
full responsibility for the content of the published article.

\bibliography{bib-exactness}
\bibliographystyle{plain}

\end{document}